\documentclass[11pt]{article}
\usepackage{amssymb, amsmath, amsthm, amsfonts,enumerate}
\usepackage{amsmath,setspace,scalefnt}
\usepackage{graphicx,tikz,caption,subcaption}
\usetikzlibrary{patterns}
\usetikzlibrary{shapes}
\usepackage{fullpage}
\usepackage{hyperref}
\usepackage{enumitem,verbatim}
\usetikzlibrary{decorations.pathmorphing}
\usepackage{subcaption}

\tikzset{snake it/.style={decorate, decoration=snake}}
\usetikzlibrary{calc}

\definecolor{gray}{rgb}{0.25, 0.25, 0.25}

\newtheorem{theorem}{Theorem}[section]
\newtheorem{lemma}[theorem]{Lemma}
\newtheorem{cor}[theorem]{Corollary}

\theoremstyle{definition}

\theoremstyle{plain}

\newtheorem{prop}[theorem]{Proposition}
\theoremstyle{definition}

\theoremstyle{definition}

\theoremstyle{definition}

\theoremstyle{definition}
\newtheorem{defn}[theorem]{Definition}
\theoremstyle{definition}

\theoremstyle{definition}

\newcommand{\ep}{\varepsilon}
\newcommand{\eps}{\varepsilon}

\newcommand{\de}{\delta}

\newcommand{\cA}{\mathcal{A}}
\newcommand{\cP}{\mathcal{P}}

\newcommand{\N}{\mathbb{N}}

\newcommand{\tk}{\mathsf{TK}}

\newcounter{propcounter}

\newcommand{\makenote}[2]
{

\smallskip

\noindent

\fbox{
\begin{minipage}{0.95\textwidth}

\def\temp{#1}
\ifx\temp\empty
\def\forlabel{$\bullet$}
\else
\def\forlabel{{\bfseries #1:}}
\fi

\begin{itemize}[label = \forlabel]
  #2/
\end{itemize}
\end{minipage}
}

\smallskip
}

\title{Balanced subdivisions of a large clique in graphs with high average degree}

 \author{
Yan Wang
\footnote{School of Mathematical Sciences, CMA-Shanghai, Shanghai Jiao Tong University, Shanghai 200240, China
(yan.w@sjtu.edu.cn)} 
\thanks{Partially supported by National Key R\&D Program of China under Grant No. 2022YFA1006400, National Natural Science Foundation of China under grant No.12201400 and Explore X project of Shanghai Jiao Tong University}
}

\date{February 8, 2023}

\begin{document}
\maketitle

\begin{abstract} 

In 1984, Thomassen conjectured that for every constant $k \in \mathbb{N}$, there exists $d$ such that every graph with average degree at least $d$ contains a balanced subdivision of a complete graph on $k$ vertices, i.e. a subdivision in which each edge is subdivided the same number of times. Recently, Liu and Montgomery confirmed Thomassen's conjecture. 

We show that for every constant $0<c<1/2$, every graph with average degree at least $d$ contains a balanced subdivision of a complete graph of size at least $\Omega(d^{c})$.
Note that this bound is almost optimal.
Moreover, we show that every sparse expander with minimum degree at least $d$ contains a balanced subdivision of a complete graph of size at least $\Omega(d)$.
\end{abstract}

\section{Introduction}\label{sec:intro}

Let $G$ be a graph. 
A \textit{subdivision} of $G$, denoted by $\mathsf{T}G$, is a graph obtained from $G$ by replacing each of its edges into internally vertex disjoint paths. 
We call the vertices of $\mathsf{T}G$ corresponding to the vertices of $G$ its \textit{core} vertices.  
Subdivisions play an important role in topological graph theory.
In 1930s, Kuratowski \cite{Kur30} showed that a graph is not planar if and only if it contains a subdivision of a complete graph on five vertices or a subdivision of a complete bipartite graph with three vertices in each partition.

For integer $t > 0$, let $d(t)$ be the minimum number $d$ such that every graph with average degree at least $d$ contains a subdivision of a complete graph $K_t$.
In 1967, Mader \cite{Mad67} showed such $d(t)$ must exist. 
Mader \cite{Mad67}, and independently Erd\H{o}s and Hajnal \cite{EH69} conjectured that $d(t) = O(t^2)$.
Later Mader \cite{Mad72} improved the upper bound of $d(t)$ to $O(2^t)$.
In 1990s, Koml\'os and Szemer\'edi \cite{K-Sz-1, K-Sz-2}, and independently, Bollob\'as and Thomason \cite{BT98} confirmed this conjecture. 
Indeed, $d(t) = \Theta(t^2)$.
As Jung \cite{Jung70} first observed, the lower bound of $d(t)$ can be achieved by disjoint union of complete regular bipartite graphs.

To guarantee a subdivision of a complete graph of size linear to the average degree, one must impose some additional conditions to eliminate the extremal examples.
Minimum girth condition is one of them as complete bipartite graphs contain many short cycles. 
In fact, Mader \cite{Mad99} conjectured that every $C_4$-free graph of average degree $d$ contains a subdivision of a complete graph of size linear to $d$.
K\"uhn and Osthus \cite{KO02, KO06} showed that every graph with sufficiently large girth contains a subdivision of a complete graph of size larger than its minimum degree. 
They \cite{KO04} also showed that every $C_4$-free graph of average degree $d$ contains a $\tk_{d/\log^{12} d}$.
In \cite{BLS15}, Balogh, Liu and Sharifzadeh proved Mader's conjecture when the graph is $C_6$-free.
Recently, Liu and Montgomery \cite{liu2017proof} completely resolved this conjecture. 
Note that those proofs utilize the technique developed by  Koml\'os and Szemer\'edi \cite{K-Sz-1, K-Sz-2}.

For $\ell \in \mathbb{N}$, an \textit{$\ell$-balanced subdivision} of $G$, denoted by $\mathsf{T}G^{(\ell)}$, is a graph obtained from $G$ by replacing each of its edges into internally vertex disjoint paths of length exactly $\ell$. 
A \textit{balanced subdivision} is an \textit{$\ell$-balanced subdivision} for some $\ell \in \mathbb{N}$.
Thomassen \cite{Tho84, Tho85, thomassen} conjectured that for every constant $k \in \mathbb{N}$, there exists $d$ such that every graph with average degree at least $d$ contains a $\tk^{(\ell)}_k$ for some $\ell \in \mathbb{N}$.
More recently, Liu and Montgomery \cite{liu2020proof} confirmed Thomassen's conjecture.

In this paper, we study the following question: 
Given a graph with average degree $d$, what is the largest size of a balanced subdivision of a complete graph that it contains as subgraph?
We will show the following. 

\begin{theorem} \label{thm:main}
For every constant $0<c<1/2$, 
every graph with average degree at least $d$ contains a 
$\tk^{(\ell)}_{\Omega(d^{c})}$ for some $\ell \in \mathbb{N}$.
\end{theorem}

We observe that an upper bound of this problem is $O(d^{1/2})$, given by disjoint union of complete regular bipartite graphs.
Therefore, Theorem \ref{thm:main} is almost optimal.

The proof of Theorem \ref{thm:main} uses the ideas from \cite{K-Sz-2, liu2017proof, liu2020proof}. 
By a result of Koml\'os and Szemer\'edi, we can find a graph that is as dense as the original graph and has some expansion property. 
Then we divide into two cases depending on whether the graph is dense or not.
The dense case is handled in Lemma \ref{lem:dense} and the sparse case is covered in Lemma \ref{lem:sparse}.

The rest of the paper will be organized as follows:
In Section \ref{sec:pre}, we introduce the notion of graph expanders and adjusters, and give some lemmas that will be used later. 
We construct large balanced clique subdivisions in dense graphs and show Lemma \ref{lem:dense} in Section \ref{sec:dense}.
The proof of Lemma \ref{lem:sparse} will be given in Section \ref{sec:sparse} where we further divide into cases according to whether such graph has many vertices of large degree or not.
We conclude in Section \ref{sec:proof}.

\subsection{Notations} \label{sec:notation}

Let $G$ be a graph. 
Let $V(G)$ and $E(G)$ be vertex set and edge set of $G$ respectively.
Let $d(G), \delta(G), \Delta(G)$ be average degree, minimum degree and maximum degree of $G$ respectively.
For $v \in V(G)$, let $d_G(v)$ denote the degree of $v$ in $G$.
We omit the subscript if there is no confusion.

For two vertices $u,v \in V(G)$, a $u,v$-path is a path with end vertices $u$ and $v$. 
We use $\ell(P)$ to denote the length of a path $P$.
The distance between two set of vertices $U,V$ in a graph $G$ is the minimum length of a $u,v$-path in $G$ with $u \in U$ and $v \in V$.
For a $u,v$-path $P$ and an integer $0 \le r \le \ell(P)$, let $P(v,r)$ be the subpath of $P$ of length $r$ with one of its end vertex to be $v$.

Let $X \subseteq V(G)$, we write $G - X$ for the induced subgraph of $G[V(G) \backslash X]$.
Denote $N_G(X)$ the (external) neighborhood of $X$ in $G - X$. 
For integer $i \ge 0$, we define the $i$-th ball around $X$ in $G$ to be the set of vertices that have distance at most $i$ from $X$ in $G$, denoted by  $B^i_G(X)$.
For convenience, $B_G(X) = B^1_G(X) = X \cup N_G(X)$.

We omit the floors and ceilings when they are not crucial.
All logarithms are natural.

\section{Preliminaries} \label{sec:pre}

\subsection{Koml\'os-Szemer\'edi graph expanders}\label{sec:expander}
The well connectedness of a graph can be measured by the expansion property. 
One form of the expansion property is as follow: For a graph $G$ and for every not too large set $X\subseteq V(G)$, $|N_G(X)|\geq \eps(|X|) |X|$ holds for some function $\eps$ depending on $|X|$.
A graph that satisfies the expansion property is called an expander graph.
A detailed coverage of expander graphs and their applications in theoretical computer science is presented in \cite{HooryLiWi06}, and their applications in mathematics are given in \cite{Lubotzky12}.

While the linear expansion property (when $\eps$ is a constant function) has been studied extensively, 
Koml\'os and Szemer\'edi~\cite{K-Sz-1,K-Sz-2} introduced sublinear expansion property, which forms the base of our proof. 

\begin{defn}
For each $\ep_1>0$ and $k>0$, a graph $G$ is an \emph{$(\eps_1,k)$-expander} if
$$|N_G(X)|\geq \eps(|X|,\eps_1,k)\cdot |X|$$
for all $X\subseteq V(G)$ with $k/2\leq |X|\leq |V(G)|/2$, where
\begin{eqnarray}\label{epsilon}
\ep(x,\ep_1,k):=\left\{\begin{tabular}{ l l }
$0$ & $\mbox{ if } x<k/5$, \\
$\ep_1/\log^2(15x/k)$ & $\mbox{ if } x\ge k/5$. \\
\end{tabular}
\right.
\end{eqnarray}
Whenever the choices of $\eps_1, k$ are clear, we omit them and write $\eps(x)$ for $\eps(x,\eps_1,k)$.
\end{defn}


In the above definition, 
note that $\ep(x,\eps_1,k)$ decreases as $x\geq k/2$ increases,
so the rate of expansion decreases as the size of $X$ grows. 
However, $\ep(x,\eps_1,k)\cdot x$ increases as $x$ increases, 
so the number of vertices that $X$ expands increases as the size of $X$ grows.

Koml\'os and Szemer\'edi~\cite{K-Sz-2} showed that every graph $G$ contains an expander subgraph with average degree and minimum degree linear to the average degree of $G$.

\begin{lemma}[\cite{K-Sz-2}]\label{thm-expander}
There exists some $\eps_1>0$ such that the following holds for every $k>0$. Every graph $G$ has an $(\eps_1,k)$-expander subgraph $H$ with $d(H)\geq d(G)/2$ and $\delta(H)\geq d(H)/2$.
\end{lemma}

Note that, in Theorem~\ref{thm-expander},  the expander subgraph $H$ can be much smaller than the original graph $G$ in size. 
To see this, one can take $G$ to be the disjoint union of many copies of such $H$.

The expansion property allows us to connect vertex sets with a short path even after removing a small set of vertices (see Lemma 3.4 from \cite{liu2017proof}).

\begin{lemma}[\cite{liu2017proof}]\label{new-connect} For each $0<\eps_1,\eps_2<1$, there exists $d_0=d_0(\eps_1,\eps_2)$ such that the following holds for each $n\geq d\geq d_0$ and $x\geq 1$. Let $G$ be an $n$-vertex $(\ep_1,\ep_2d)$-expander with $\delta(G)\geq d-1$.

Let $A,B\subseteq V(G)$ with $|A|,|B|\geq x$, and  let $W\subseteq V(G)\setminus(A\cup B)$ satisfy $|W|\log^3n\leq 10x$. Then, there is a path from $A$ to $B$ in $G-W$ with length at most $\frac{40}{\eps_1}\log^3n$.
\end{lemma}

It is well known that every graph $G$ has a bipartite subgraph $H$ with $d(H)\geq d(G)/2$.
The next corollary follows immediately from this fact and Lemma~\ref{thm-expander}.

\begin{cor}\label{cor-expander}
There exists some $\eps_1>0$ such that the following holds for every $\eps_2>0$ and $d\in \N$. Every graph $G$ with $d(G)\geq 8d$ has a bipartite $(\eps_1,\eps_2d)$-expander subgraph $H$ with $\delta(H)\geq d$.
\end{cor}

The following notation (also see \cite{liu2020proof}) is convenient as we often work in a bipartite graph.

\begin{defn}[\cite{liu2020proof}] For any connected bipartite graph $H$ and $u,v\in V(H)$, let\label{pidefn}
$$
\pi(u,v,H)=\left\{\begin{array}{ll}
0 & \text{ if }u=v, \\
1 & \text{ if $u$ and $v$ are in different vertex classes in the (unique) bipartition of $H$},\\
2 & \text{ if $u$ and $v$ are in the same vertex class and $u\neq v$}.
\end{array}
\right.
$$
\end{defn}

\subsection{Liu-Montgomery adjusters}\label{sec:adjusters}

First, we need the definition below from Liu and Montgomery \cite{liu2020proof}.

\begin{defn}[\cite{liu2020proof}]
Given a vertex $v$ in a graph $F$, $F$ is a \emph{$(D,m)$-expansion of $v$} if $|F|=D$ and $v\in V(F)$ is at distance at most $m$ in $F$ from any other vertex of $F$.
\end{defn}

Expansion has the following trimming property.

\begin{prop}[\cite{liu2020proof}] \label{prop-trimming} Let $D,m\in \N$ and $1\leq D'\leq D$. Then, any graph $F$ which is a $(D,m)$-expansion of $v$ contains a subgraph which is a $(D',m)$-expansion of $v$.
\end{prop}

The following is a technical lemma that allows us to find one set within a collection that expands to larger size. 
We remark that the same proof (by setting $\alpha=2^{-10}$) of the original Lemma 3.7 in \cite{liu2020proof} could imply a stronger conclusion that there exists a set that expands to size $\exp((\log \log n)^{200})$, instead of $\log^k n$.

\begin{lemma}[\cite{liu2020proof}] \label{lem:small-set-expansion-new}
For each $0 < \eps_1 < 1$ and $0 < \eps_2 < 1/5$, there exists $d_0 = d_0(\eps_1, \eps_2)$ such that the following holds for each $n \ge d \ge d_0$.
Suppose that $G$ is an $n$-vertex bipartite $(\eps_1, \eps_2d)$-expander with $\delta(G) \ge d$.
Let $U \subseteq V(G)$ satisfy $|U| \le \exp((\log \log n)^{200})$.
Let $r = n^{1/8}$ and $\ell_0 = (\log \log n)^{10^6}$.
Let $m = \frac{800}{\eps_1} \log^3 n$. 
Suppose $(A_i,B_i,C_i), i \in [r]$ are such that the following holds for each $i \in [r]$. 
  \stepcounter{propcounter}
  \begin{enumerate}[label = {\bfseries \Alph{propcounter}\arabic{enumi}}]
  \item $|A_i| \ge d_0$.
  \item $B_i \cup C_i$ and $A_i$ are disjoint sets in $V(G) \setminus U$, with $|B_i| \le |A_i|/\log^{1000}|A_i|$.
  \item $A_i$ has $4$-limited contact with $C_i$ in $G - U - B_i$.
  \item Each vertex in $B^{\ell_0}_{G-U-B_i-C_i}(A_i)$ has at most $d/2$ neighbours in $U$.
  \item For each $j \in [r]\setminus \{i\}$, $A_i$ and $A_j$ are at least at distance $2\ell_0$ apart in $G-U-B_i-C_i-B_j-C_j$.
\end{enumerate}
Then, for some $i \in [r]$, $|B^{\ell_0}_{G-U-B_i-C_i}(A_i)| \ge \exp((\log \log n)^{200})$.
\end{lemma}

Liu and Montgomery \cite{liu2020proof} introduced a structure called \textit{adjuster} which contains paths of lengths that belong to a long arithmetic progression of difference $2$. 
We can use this structure to adjust a path to the desired length. 

\begin{defn}[\cite{liu2020proof}] \label{defn-adj}
	A \emph{$(D,m,k)$-adjuster} $\cA=(v_1,F_1,v_2,F_2,A)$ in a graph $G$ consists of vertices $v_1,v_2\in V(G)$, graphs $F_1,F_2\subseteq G$ and a vertex set $A\subseteq V(G)$ such that the following hold for some $\ell\in\N$.
  \stepcounter{propcounter}
  \begin{enumerate}[label = {\bfseries \Alph{propcounter}\arabic{enumi}}]
  \item $A$,  $V(F_1)$ and $V(F_2)$ are pairwise disjoint.\label{d-a-1}
  \item For each $i\in [2]$, $F_i$  is a $(D,m)$-expansion of $v_i$.\label{d-a-2}
  \item $|A|\leq 10mk$.\label{d-a-3}
  \item For each $i\in \{0,1,\ldots,k\}$, there is a $v_1,v_2$-path in $G[A\cup \{v_1,v_2\}]$ with length $\ell+2i$.\label{d-a-4}
\end{enumerate}
We call the smallest such $\ell$ for which these properties hold the \emph{length of the adjuster} and denote it $\ell(\cA)$. Note that it immediately follows that $\ell(\cA)\leq |A|+1\leq 10mk+1$. We call a $(D,m,1)$-adjuster a \emph{{simple adjuster}}. 
Let $V(\cA)=V(F_1)\cup V(F_2)\cup A$.
\end{defn}

Note that if $m > m'$, then a $(D,m',k)$-adjuster is also a $(D,m,k)$-adjuster.
Lemma 4.3 in \cite{liu2020proof} showed that there exists a simple adjuster in every expander without $\tk^{(2)}_{d/2}$ even after removing a moderate size of vertices. 
We remark that Lemma \ref{lem:small-set-expansion-new} allows us to prove the following strengthening variant.
(The same proof works by setting $\Delta=\exp((\log \log n)^{400}), \ell_0=(\log \log n)^{10^6}$ and $|Z|=\exp((\log \log n)^{400})$ in the original proof.)

\begin{lemma}[\cite{liu2020proof}] \label{lem:simple-adjuster-robust-new}
There exists some $\eps_1 > 0$ such that, for any $0 < \eps_2 < 1/5$, there exists $d_0 = d_0(\eps_1, \eps_2)$ such that the following holds for each $n \ge d \ge d_0$.
Suppose that $G$ is an $n$-vertex $\tk^{(2)}_{d/2}$-free bipartite $(\eps_1, \eps_2d)$-expander with $\delta(G) \ge d$.
Let $m = \frac{200}{\eps_1} \log^3 n$ and $|D| \le \exp((\log \log n)^{200})$.
Let $U \subseteq V(G)$ such that $|U| \le \exp((\log \log n)^{200})$. 

Then, $G-U$ contains a $(D,m,1)$-adjuster.
\end{lemma}

Therefore, as Lemma 4.7 in \cite{liu2020proof}, we obtain the following variant that connects simple adjusters into a larger adjuster robustly.

\begin{lemma}[\cite{liu2020proof}] \label{lem:large-adjuster}
There exists some $\eps_1 > 0$ such that, for any $0 < \eps_2 < 1/5$, there exists $d_0 = d_0(\eps_1, \eps_2)$ such that the following holds for each $n \ge d \ge d_0$.
Suppose that $G$ is an $n$-vertex $\tk^{(2)}_{d/2}$-free bipartite $(\eps_1, \eps_2d)$-expander with $\delta(G) \ge d$.

Let $m = \frac{800}{\eps_1} \log^3 n$. 
Suppose $\log^{10} n \le D \le \exp((\log \log n)^{100})$, $1 \le r \le 20m$ and $U \subseteq V(G)$ with $|U| \le \exp((\log \log n)^{100})$.

Then, there is a $(D,m,r)$-adjuster in $G - U$.
\end{lemma}

\subsection{Connecting vertices by paths of specific lengths}

Liu and Montgomery (see Corollary 3.15 in \cite{liu2020proof}) proved the existence of two vertex disjoint paths in an expander graph so that the sum of their lengths is close to the desired length while avoiding a set of moderate size. 

\begin{lemma}[\cite{liu2020proof}]\label{longconnect4}
For any $0<\eps_1, \eps_2<1$, there exists $d_0=d_0(\eps_1,\eps_2)$ such that the following holds for each $n\geq d\geq d_0$. Suppose that $G$ is an $n$-vertex bipartite $(\eps_1,\eps_2 d)$-expander with $\de(G)\ge d$.

Let $\log^{10}n\leq D\leq n/\log^{10}n$, $\frac{100}{\ep_1}\log^3n\le m\le \log^4n$ and $\ell\leq n/\log^{10}n$. Let $A\subseteq V(G)$ satisfy $|A|\leq D/\log^3n$. Let $F_1,\ldots,F_4\subseteq G-A$ be vertex disjoint subgraphs and $v_1,\ldots,v_4$ be vertices such that, for each $i\in [4]$, $F_i$ is a $(D,m)$-expansion of $v_i$.

Then, $G-A$ contains vertex disjoint paths $P$ and $Q$ with $\ell \leq \ell(P)+\ell(Q)\leq \ell+20m$ such that both $P$ and $Q$ connect $\{v_1,v_2\}$ to $\{v_3,v_4\}$.
\end{lemma}

Using adjuster structure, they also showed the existence of a path of specific length connecting two given vertices while avoiding a set of moderate size (see Lemma 4.8 in \cite{liu2020proof}). 

\begin{lemma}[\cite{liu2020proof}]\label{lem-finalconnect-sparse}
There exists some $\eps_1>0$ such that, for any $0<\eps_2<1/5$ and $k\geq 10$, there exists $d_0=d_0(\eps_1,\eps_2,k)$ such that the following holds for each $n\geq d\geq d_0$. Suppose that $G$ is an $n$-vertex $\tk_{d/2}^{(2)}$-free bipartite $(\eps_1,\eps_2 d)$-expander with $\de(G)\ge d$.

Suppose $\log^{10} n\leq D\leq \log^kn$, and $U\subseteq V(G)$ with $|U|\leq D/2\log^3n$, and let $m=\frac{800}{\eps_1}\log^3n$. Suppose $F_1,F_2\subseteq G-U$ are vertex disjoint such that $F_i$ is a $(D,m)$-expansion of $v_i$, for each $i\in[2]$. Let $\log^{7}n\leq \ell\leq n/\log^{10}n$ be such that $\ell=\pi(v_1,v_2,G)\mod 2$.

Then, there is a $v_1,v_2$-path with length $\ell$ in $G-U$.
\end{lemma}

\subsection{Expansion of vertices in sparse graphs}

When maximum degree is bounded, there exist many vertices that are pairwise far apart in the graph (see Proposition 5.3 in \cite{liu2017proof}). 

\begin{lemma}[\cite{liu2017proof}] \label{lem-sparse-far-apart}
Let $s \ge 1$. 
There exists $n_0$ such that the following holds for all $n \ge n_0$.
Suppose $G$ is an $n$-vertex graph with maximum degree at most $\log^{30s} n$. 
Then $G$ contains at least $n^{1/5}$ vertices which are pairwise at distance at least $\log n/(50s \log\log n)$ apart.
\end{lemma}

We need the following definition.

\begin{defn}[\cite{liu2017proof}]
We say that paths $P_1,\cdots,P_q$, each starting with the vertices $v$ and contained in the vertex set $W$, are \textit{consecutive shortest paths} from $v$ in $W$ if, for each $i$, $1 \le i \le q$, the path $P_i$ is a shortest path between its endpoints in the set $W - \cup_{j < i} P_j + v$.
\end{defn}

We can expand a vertex $v$ to a set of moderate size while avoiding internal vertices of a family of paths if those paths do not intersect balls around $v$ much (see Lemma 5.5 in \cite{liu2017proof}). 
Note that in the following lemma, the condition that $P_1, \cdots, P_q$ are consecutive shortest paths from $v$ in $B_H^r(v)$,
is used to guarantee that for each $p < r$, only the first $p + 2$ vertices of each of the paths $P_i$ including the vertex $v$, can belong in $N_H(B^p_{H-P+v}(v))$.

\begin{lemma}[\cite{liu2017proof}] \label{lem-sparse-small-ball}
Let $0 < \eps_1 < 1, 0 < \eps_2 < 1/20$ and $s \ge 1$.
Then there is some $c>0$ and $d_0 \in \mathbb{N}$ for which the following holds for any $n$ and $d$ with $d_0 \le d \le \log^{20s} n$.
Suppose $H$ is an $n$-vertex $(\eps_1, \eps_2 d)$-expander with $\delta(H) \ge d/16$. 
Let $r=(\log\log n)^5$ and $P = \cup_i V(P_i)$, 
if $q \le cd$ and $P_1, \cdots, P_q$ are consecutive shortest paths from $v$ in $B_H^r(v)$,
then $|B^r_{H-P+v}(v)| \ge 2d^2 \log^{10} n$.
\end{lemma}

We can further expand a set of moderate size to a large set avoiding a set of vertices (see Proposition 5.6 in \cite{liu2017proof}).

\begin{lemma}[\cite{liu2017proof}] \label{lem-sparse-large-ball}
Let $0 < \eps_1 < 1$, $0 < \eps_2 < 1/20$ and $s \ge 1$.
There exists $d_0 = d_0(\eps_1, \eps_2, s)$ such that the following is true for each $d \ge d_0$.
Suppose that $H$ is an $(\eps_1, \eps_2 d)$-expander with $n$ vertices and let $k = \log n/100s \log \log n$.
If $Y, W \subseteq V(G)$ are disjoint sets with $|Y| \ge d^2 \log^{10} n$ and $|W| \le d^2 \log^7 n$,
then $|B^k_{G-W}(Y)| \ge \exp((\log n)^{1/4})$.
\end{lemma}

Note that the exponents of $\log n$ in the above two lemmas are altered to adapt to our proof.
The proofs are similar to the original proofs so we omit here.

\section{Constructing balanced clique subdivisions in dense graphs}\label{sec:dense}

In this section, we deal with the case when graph is dense, that is, $d \ge \log^{s} n$ for some large $s$, and prove the following.

\begin{lemma} \label{lem:dense}
There exists $0 < \eps_1 < 1$ such that for any $0 < \varepsilon_2 < 1$ and $s \ge 20$,
there exists $d_0 = d_0(\varepsilon_1, \varepsilon_2, s)$ such that the following holds for each $n \ge d \ge d_0$ and $d \ge \log^{s} n$.
Suppose that $G$ is an $n$-vertex bipartite $(\varepsilon_1, \varepsilon_2 d)$-expander with $\delta(G) \ge d$.
Then $G$ contains a $\tk^{(\ell)}_{\sqrt{d}/2\log^{10} n}$ for some $\ell \in \mathbb{N}$.
\end{lemma}

Note that when such graph $G$ is dense, it contains a balanced clique subdivision of size at least $\Omega(d^{\frac{1}{2}-\frac{10}{s}})$ for arbitrarily large $s$.

We adopt the idea in \cite{liu2020proof} that uses adjuster structure to alter the length of a path, and refine the analysis when graph is dense. 
In any dense expander graph with minimum degree $d$, we can find a simple adjuster with expansion size linear to $d$.

\begin{lemma}\label{lem-simple-adjuster}
	For any $\eps_1>0$, $\eps_2>0$ and $s \ge 20$, there exists $d_0=d_0(\eps_1,\eps_2, s)$ such that the following is true for each $n\geq d\geq d_0$ and $d \ge \frac{1}{10} \log^s n$. Suppose that $G$ is an $n$-vertex bipartite $(\eps_1,\eps_2 d)$-expander with $\de(G)\ge d$.
  Let $D\leq d/3$.

  Then, $G$ contains a $(D,\log n,1)$-adjuster.
\end{lemma}

\begin{proof}
Let $C$ be a shortest cycle in $G$. 
Since $G$ is bipartite, $C$ must have even length, say $2\ell_0$.
Since $\delta(G) \ge d$, we have $\ell_0 \le \log n / \log (d-1) < \log n$.
Let $x_1, x_2$ be two vertices of distance $\ell_0 - 1$ on $C$. 

Choose $F_i \subseteq N(x_i)$ such that $|F_i| = D-1$ and $F_i$ is disjoint from $F_{3-i} \cup V(C) \cup \{x_1, x_2\}$ for $i \in [2]$.
This is possible because we may first choose $F_1 \subseteq N(x_1) \setminus (V(C) \cup \{x_2\})$ and then choose $F_2 \subseteq N(x_2) \setminus (F_1 \cup V(C) \cup \{x_1\})$, noting that $|N(x_2)| - |F_1| - |V(C)| > d - d/3 -\log n > d/3$.
Thus $F_i \cup \{x_i\}$ is $(D,2)$-expansion of $x_i$ for $i \in [2]$.
Let $A = C \backslash \{x_1,x_2\}$, so $|A| < |C| < 2 \log n$. 
Therefore, $(x_1,F_1 \cup \{x_1\},x_2,F_2 \cup \{x_2\},A)$ is a desired $(D,\log n,1)$-adjuster.
\end{proof}

Moreover, such a simple adjuster exists robustly in dense expander graph upon removal of any subset of vertices of moderate size.

\begin{lemma}\label{lem-simple-adjuster-robust}
	There exists some $\eps_1>0$ such that for every $\eps_2>0$ and $s \ge 20$, there exists $d_0=d_0(\eps_1,\eps_2, s)$ such that the following is true for each $n\geq d\geq d_0$ and $d \ge \log^s n$. Suppose that $G$ is an $n$-vertex bipartite $(\eps_1,\eps_2 d)$-expander with $\de(G)\ge d$.
Let $D \le d/30$ and $U \subseteq V(G)$ such that $|U| \le d/10$.

Then, $G - U$ contains a $(D,\log n,1)$-adjuster.
\end{lemma}

\begin{proof}
Let $\eps_1 > 0$ be such that Corollary \ref{cor-expander} holds.
Note that 
$d(G - U) = 2|E(G-U)| / |V(G-U)| \ge (dn - 2n|U|)/n = d-2|U| \ge 4d/5$.
 
By Corollary \ref{cor-expander}, $G - U$ has a bipartite $(\varepsilon_1, \varepsilon_2 d/10)$-expander subgraph $H$ with $\delta(H) \ge d/10$.
Also note that $d(H) \ge \delta(H) \ge d/10 \ge \frac{1}{10} \log^s n \ge \frac{1}{10} \log^s |V(H)|$.
Apply Lemma \ref{lem-simple-adjuster} with $G_{\ref{lem-simple-adjuster}} = H$, we obtain a $(D,\log n,1)$-adjuster in $H$, and thus in $G-U$.
\end{proof}

We chain simple adjuster together into a larger adjuster. 

\begin{lemma}\label{lem-adj-to-adj-path}
	There exists some $0<\eps_1<1$ such that for every $0<\eps_2<1$ and $s \ge 20$, there exists $d_0=d_0(\eps_1,\eps_2, s)$ 
such that the following is true for each $n\geq d\geq d_0$ and $d \ge \log^s n$. Suppose that $G$ is an $n$-vertex bipartite $(\eps_1,\eps_2 d)$-expander with $\de(G)\ge d$.
Let $c = 1/100$, $m = \frac{800}{\eps_1} \log^3 n$, $D = cd$, $1 \le r \le 20m$ and $U \subseteq V(G)$ such that $|U| \le cd/10$.

Then, $G - U$ contains a $(D,m,r)$-adjuster.
\end{lemma}

\begin{proof}
Let $\eps_1 > 0$ be such that Lemma \ref{lem-simple-adjuster-robust} holds and $d_0 = d(\eps_1, \eps_2, s)$ be large. 
We prove the property by induction on $r$.
When $r = 1$, Lemma \ref{lem-simple-adjuster-robust} gives the desired $(D,\log n,1)$-adjuster.

Now assume that for some $1 \le r < 20m$, $G - U$ contains a $(D,m,r)$-adjuster, say $\cA_1:=(v_1, F_1, v_2, F_2, A_1)$.
We aim to show the existence of a $(D,m, r+1)$-adjuster.
Let $U' = U \cup F_1 \cup F_2 \cup A_1$.
So $|U'| = |U| + |F_1| + |F_2| + |A_1| \le cd/10 + cd + cd + 10mr < 3cd$.
Applying Lemma \ref{lem-simple-adjuster-robust} with $(G,U)_{\ref{lem-simple-adjuster-robust}} = (G,U')$, we have that $G - U'$ contains a $(D,\log n,1)$-adjuster, say $\cA_2:=(v_3, F_3, v_4, F_4, A_2)$.
Since $|F_1 \cup F_2| = |F_3 \cup F_4| = 2D$ and $|A_1 \cup A_2| < 20 mr  <  \log^7 n < 2D / (10\log^3 n)$,
there exists a path $P$ of length at most $m$ from $V(F_1) \cup V(F_2)$ to $V(F_3) \cup V(F_4)$ avoiding $A_1 \cup A_2$ by Lemma \ref{new-connect}.

Without loss of generality, we may assume that $P$ is from $V(F_1)$ to $V(F_3)$. 
By definition of $F_1$ and $F_3$, there exists a path $Q$ from $v_1$ to $v_3$ in $F_1 \cup P \cup F_3$ of length at most $2\log n + 3m$.
We claim that $\cA_3:=(v_2, F_2, v_4, F_4, A_1 \cup A_2 \cup V(Q))$ is a desired $(D,m,r+1)$-adjuster.
In fact, by construction $F_2$, $F_4$ and $A_1 \cup A_2 \cup V(Q)$ are pairwise disjoint, and $F_i$ is a $(D, m)$-expansion of $v_i$ for $i \in \{2,4\}$.
Note that $|A_1 \cup A_2 \cup V(Q)| \le |A_1| + |A_2| + |V(Q)| \le 10mr + 10 \log n + (2\log n + 3m) \le 10m(r+1)$.
Finally, let $\ell:=\ell(\cA_1)+\ell(\cA_2)+|V(Q)|$.
We show that for each $i \in \{0,1,\cdots, r+1\}$, there is a $v_2,v_4$-path in $G[A_1 \cup A_2 \cup V(Q) \cup \{v_2, v_4\}]$ with length $\ell+2i$.
If $i \in \{0,1,\cdots,r\}$, then let $i_1=i$ and $i_2=0$; otherwise, let $i_1=r$ and $i_2=1$.
Let $P_1$ be a $v_2,v_1$-path of length $\ell(\cA_1) + 2i_1$ in $G[A_1 \cup \{v_1,v_2\}]$ and $P_2$ be a $v_3,v_4$-path of length $\ell(\cA_2) + 2i_2$ in $G[A_2 \cup \{v_3,v_4\}]$.
Therefore, $P_1 \cup Q \cup P_2$ is a desired $v_2,v_4$-path in $G[A_1 \cup A_2 \cup V(Q) \cup \{v_2, v_4\}]$ with length $\ell+2i$.
\end{proof}

In the follow lemma, we show that there exists a path of certain length connecting two given vertices. 

\begin{lemma}\label{lem-finalconnect-dense}
There exists some $0 < \eps_1 < 1$ such that, for any $0<\eps_2<1$ and $s \ge 20$, there exists $d_0=d_0(\eps_1,\eps_2, s)$ such that the following holds for each $n\geq d\geq d_0$ and $d \ge \log^s n$. Suppose that $G$ is an $n$-vertex bipartite $(\eps_1,\eps_2 d)$-expander with $\de(G)\ge d$.

Suppose $D=d/\log^{10} n$, and $U\subseteq V(G)$ with $|U|\leq D/2\log^3n$, and let $m=\frac{800}{\eps_1}\log^3n$. Suppose $F_1,F_2\subseteq G-U$ are vertex disjoint subgraphs such that $F_i$ is a $(D,m)$-expansion of $v_i$, for each $i\in[2]$. Let $\log^{7}n\leq \ell\leq n/\log^{10}n$ be such that $\ell=\pi(v_1,v_2,G)\mod 2$.

Then, there is a $v_1,v_2$-path with length $\ell$ in $G-U$.
\end{lemma}

\begin{proof}
Let $0 < \eps_1 < 1$ be such that Lemma \ref{lem-adj-to-adj-path} holds and $d_0 = d_0(\eps_1, \eps_2, s)$ be large.

Let $U' = U \cup V(F_1) \cup V(F_2)$.
So $|U'| = |U| + |V(F_1)| + |V(F_2)| \le D/(2 \log^3 n) + D + D \le 3D < d / 1000$. 
By Lemma~\ref{lem-adj-to-adj-path} with $(G,U,r)_{\ref{lem-adj-to-adj-path}} = (G,U',20m)$, there is a $(d/100,m,20m)$-adjuster, and thus a $(D,m,20m)$-adjuster, say $\cA=(v_3,F_3,v_4,F_4,A)$, in $G-U'$ with length $\ell(\cA)\le |A|+1\leq 400m^2$. Let $\bar{\ell}=\ell-20m-\ell(\cA)$,
so that $0\leq \bar{\ell}\le n/\log^{10}n$. As $|A\cup U|\leq 400m^2+D/2\log^3n\leq D/\log^3n$, by Lemma~\ref{longconnect4} with $(G,D,A)_{\ref{longconnect4}} = (G,D,A \cup U)$, there are paths $P$ and $Q$ in $G-U-A$ which are vertex disjoint, both connect $\{v_1,v_2\}$ to $\{v_3,v_4\}$ and so that $\bar{\ell}\leq \ell(P)+\ell(Q)\leq \bar{\ell}+20m$.
Note that we can assume, without loss of generality, that $P$ is a $v_1,v_3$-path and $Q$ is a $v_2,v_4$-path.

Now, $0\leq \ell-\ell(P)-\ell(Q)-\ell(\cA)\leq 20m$. 
As $\cA$ is a $(D,m,20m)$-adjuster, there is a $v_3,v_4$-path in $G[A\cup \{v_3,v_4\}]$ with length $\ell(\cA)$, and therefore $\ell(\cA)= \pi(v_3,v_4,G)\mod 2$. 
Then, as $\ell(P)=\pi(v_1,v_3,G)\mod 2$, $\ell(Q)=\pi(v_2,v_4,G)\mod 2$, $\ell=\pi(v_1,v_2,G)\mod 2$ and $\pi(v_1,v_2,G)=\pi(v_1,v_3,G)+\pi(v_3,v_4,G)+\pi(v_4,v_2,G) \mod 2$, we have $\ell-\ell(P)-\ell(Q)-\ell(\cA)=0\mod 2$. That is, there is some $i\in \N$ with $2i=\ell-\ell(P)-\ell(Q)-\ell(\cA)$, where $i\leq 10m$.

Therefore, by the definition of the adjuster, there is a $v_3,v_4$-path $R$  with length $\ell(\cA)+2i=\ell-\ell(P)-\ell(Q)$ in $G[A\cup\{v_3,v_4\}]$. Then, $P\cup R\cup Q$ is a $v_1,v_2$-path with length $\ell$ in $G-U$.
\end{proof}

Finally, we are ready to show Lemma \ref{lem:dense}.

\begin{proof}[Proof of Lemma \ref{lem:dense}]
Let $0 < \eps_1 < 1$ be such that Lemma \ref{lem-finalconnect-dense} holds and $d_0 = d_0(\eps_1, \eps_2, s)$ be large. 
Let $t = \lfloor \sqrt{d}/(2 \log^{10} n) \rfloor$ and $\ell = 2 \lceil \log^7 n \rceil \equiv 0 \pmod 2$. 
Let $v_1,\cdots,v_t$ be $t$ distinct vertices in the same class of the partition of $G$ that will serve as $t$ core vertices in the final balanced clique subdivision construction.
Let $\mathcal{P} = \{P_1,\cdots,P_K\}$ be a maximal collection of pairwise internally disjoint paths such that
\stepcounter{propcounter}
\begin{enumerate}[label = {\bfseries \Alph{propcounter}\arabic{enumi}}]
\item For each $k \in [K]$, $P_k$ is a $v_i,v_j$-path of length $\ell$ for some distinct $i,j \in [t]$.
\item For distinct $i,j \in [t]$, there is at most one path in $\mathcal{P}$ with $v_i$ and $v_j$ as end vertices. 
\end{enumerate}

One can verify that if $K = {t \choose 2}$, then the graph formed by all the paths in $\mathcal{P}$ is a desired $\tk^{(\ell)}_t$. 
Hence, we may assume that there exist distinct $i,j \in [t]$ such that $\mathcal{P}$ contains no such $v_i,v_j$-path of length $\ell$.

Let $U = (\bigcup_{k \in [K]} V(P_k)) \bigcup \{v_q: q \in [t]\} \backslash \{v_i,v_j\}$.
So $|U| \le K \ell + t \le t^2 \ell \le d / 2 \log^{13} n$.
Choose $F_m \subseteq N(v_m)$ such that $|F_m| = d/\log^{10}n - 1$ and $F_m$ is disjoint from $F_{i+j-m} \cup U \cup \{v_i, v_j\}$ for $m \in \{i,j\}$.
This is possible because we may first choose $F_i \subseteq N(v_i) \setminus (U \cup \{v_j\})$ and then choose $F_j \subseteq N(v_j) \setminus (F_i \cup U \cup \{v_i\})$, noting that $|N(v_j)| - |F_i| - |U| > d - d/\log^{10}n - d/\log^{13}n > d/\log^{10}n$.
By Lemma \ref{lem-finalconnect-dense}, there is a $v_i,v_j$-path $P_{K+1}$ of length $\ell$ in $G - U$.
Therefore, $\{P_1, P_2, \cdots, P_{K+1}\}$ contradicts the maximality of $\mathcal{P}$.
This completes the proof.
\end{proof}

\section{Constructing balanced clique subdivisions in sparse graphs}\label{sec:sparse}

In this section, we handle the case when graph is sparse, that is $d < \log^s n$ for some large $s$, and prove the following.

\begin{lemma} \label{lem:sparse}
There exists $\eps_1 > 0$ such that for any $0 < \varepsilon_2 < 1/5$ and $s \ge 20$,
there exist $d_0 = d_0(\varepsilon_1, \varepsilon_2, s)$ and some constant $t > 0$ such that the following holds for each $n \ge d \ge d_0$ and $d < \log^{s} n$.
 Suppose that $G$ is a $\tk_{d/2}^{(2)}$-free $n$-vertex bipartite $(\varepsilon_1, \varepsilon_2 d)$-expander with $\delta(G) \ge d$.
Then $G$ contains a $\tk^{(\ell)}_{t d}$ for some $\ell \in \mathbb{N}$.
\end{lemma}

Note that when such graph $G$ is sparse, we are able to find a balanced clique subdivision of size linear to its average degree.
This may be of independent interest.

We discuss two cases depending on whether there are many vertices with degree at least $\Delta(G) \ge c^2 d^2 \log^{10} n$ or not.

\subsection{When many vertices have large degree}

When many vertices have large degree, we could use the neighbourhood of those vertices to construct expansion and find paths of specific length between them. 

\begin{lemma} \label{lem:sparse-large-deg}
There exists $\eps_1 > 0$ such that for any $0 < \varepsilon_2 < 1/5$, $c > 0$ and $s \ge 20$,
there exist $d_0 = d_0(\varepsilon_1, \varepsilon_2, s)$ and a constant $0 < t_3 < c/3$ such that the following holds for each $n \ge d \ge d_0$ and $d < \log^{s} n$.
Suppose that $G$ is a $\tk_{d/2}^{(2)}$-free $n$-vertex bipartite $(\varepsilon_1, \varepsilon_2 d)$-expander subgraph $G$ with $\delta(G) \ge d$.
Moreover, suppose at least $2 t_3 d$ vertices have degree at least $\Delta = c^2 d^2 \log^{10} n$. 
Then $G$ contains a $\tk^{(\ell)}_{t_3 d}$ for some $\ell \in \mathbb{N}$.
\end{lemma}

\begin{proof}
Let $\eps_1 > 0$ be such that Lemma \ref{lem-finalconnect-sparse} holds and $d_0 = d_0(\eps_1, \eps_2, s)$ be large. 
Let $0 < t_3 < c/3$, $t = \lceil t_3 d \rceil$ and $\ell = 2 \lceil \log^7 n \rceil$. 
Let $v_1,\cdots,v_t$ be $t$ distinct vertices in the same partition of $G$ of degree at least $\Delta$ that will serve as $t$ core vertices in the final balanced clique subdivision construction.
Let $\mathcal{P} = \{P_1,\cdots,P_K\}$ be a maximal collection of pairwise internally disjoint paths such that
\stepcounter{propcounter}
\begin{enumerate}[label = {\bfseries \Alph{propcounter}\arabic{enumi}}]
\item For each $k \in [K]$, $P_k$ is a $v_i,v_j$-path of length $\ell$ for some distinct $i,j \in [t]$.
\item For distinct $i,j \in [t]$, there is at most one path in $\mathcal{P}$ with $v_i$ and $v_j$ as end vertices. 
\end{enumerate}

One can verify that if $K = {t \choose 2}$, then the graph formed by all the paths in $\mathcal{P}$ is a desired $\tk^{(\ell)}_t$. 
Hence, we may assume that there exist distinct $i,j \in [t]$ such that $\mathcal{P}$ contains no such $v_i,v_j$-path of length $\ell$.

Let $U = (\bigcup_{k \in [K]} V(P_k)) \bigcup \{v_q: q \in [t]\} \backslash \{v_i,v_j\}$.
So $|U| \le K \ell + t \le t^2 \ell \le t_3^2 d^2 \log^{7} n$.
Choose $F_m \subseteq N(v_m)$ such that $|F_m| = \frac{c^2}{4} d^2 \log^{10} n - 1$ and $F_m$ is disjoint from $F_{i+j-m} \cup U \cup \{v_i, v_j\}$ for $m \in \{i,j\}$.
This is possible because we may first choose $F_i \subseteq N(v_i) \setminus (U \cup \{v_j\})$ and then choose $F_j \subseteq N(v_j) \setminus (F_i \cup U \cup \{v_i\})$, noting that $|N(v_j)| - |F_i| - |U| > c^2 d^2 \log^{10} n - \frac{c^2}{4} d^2 \log^{10} n - t_3^2 d^2 \log^{7} n > \frac{c^2}{4} d^2 \log^{10} n$.
By Lemma \ref{lem-finalconnect-sparse} with $(G,U,D,k)_{\ref{lem-finalconnect-sparse}} = (G, U, \frac{c^2}{4} d^2 \log^{10} n, 2s+10)$, there is a $v_i,v_j$-path $P_{K+1}$ of length $\ell$ in $G - U$.
Therefore, $\{P_1, P_2, \cdots, P_{K+1}\}$ contradicts the maximality of $\mathcal{P}$.
This completes the proof.
\end{proof}

\subsection{When all vertices have bounded maximum degree}

We adopt the idea from \cite{K-Sz-2} and \cite{liu2017proof}.
We find two balls of radius $r_1$ and $r_2$ respectively ($r_1 \ll r_2$) around each core vertex, and try to connect two core vertices by a path of specific length avoiding balls of radius $r_1$ of all other core vertices.
The existence of such paths is guaranteed by the expander property and large expansion around each core vertex (namely, the ball of radius $r_2$).
Since the ball of radius $r_1$ of each core vertex is only used by the paths leading to it, we can also grow it to form a ball of radius $r_2$ in each step.

First, we prove a strengthened version of Lemma \ref{lem-finalconnect-sparse}.

\begin{lemma} \label{lem-finalconnect-sparse2}
There exists some $\eps_1>0$ such that, for any $0<\eps_2<1/5$ and $s \ge 20$, there exists $d_0=d_0(\eps_1,\eps_2,s)$ such that the following holds for each $n\geq d\geq d_0$ and $d < 2 \log^{s} n$. Suppose that $G$ is an $n$-vertex $\tk_{d/2}^{(2)}$-free bipartite $(\eps_1,\eps_2 d)$-expander with $\de(G)\ge d$.

Moreover, let $r = \lceil (\log \log n)^5 \rceil$ and $V_i = B_{G}^{r}(v_i)$ with $|V_i| \le \exp((\log \log n)^7)$ for $i \in [2]$.
Suppose that $\cP$ is a family of paths such that there is an ordering $\cP_i$ of consecutive shortest paths $\cP(v_i,r) := \{P(v_i,r): P \in \cP, v_i \in P\}$ from $v_i$ in $V_i$ 
for $i \in [2]$.

Suppose $U\subseteq V(G)$ such that $U \cap V_i = (V(\cP) \cap V_i) \setminus \{v_i\} $ for $i \in [2]$ and $|U|\leq \exp((\log \log n)^{10})$, and let $m=\frac{800}{\eps_1}\log^3n$. Suppose that $F_1,F_2\subseteq G - U$ are vertex disjoint subsets such that $B_{G-V(\cP) + v_i}^r(v_i) \subseteq F_i$ and $F_i$ is an $(\exp((\log \log n)^{90}),m)$-expansion of $v_i$ in $G - U$ for $i\in[2]$. Let $\log^{7}n\leq \ell\leq n/\log^{10}n$ be such that $\ell=\pi(v_1,v_2,G)\mod 2$.

Then, there is a $v_1,v_2$-path $P'$ with length $\ell$ in $G-U$.
Moreover, we have $\cP_i, P'(v_i,r)$ is an ordering of consecutive shortest paths from $v_i$ in $V_i$ 
and $(V(P') \setminus V(P'(v_i,r)) ) \cap B^{r}_{G - V(\cP) + v_i}(v_i) = \emptyset$
for $i \in [2]$.
\end{lemma}

\begin{proof}
Let $0 < \eps_1 < 1$ be such that Lemma~\ref{lem:large-adjuster} holds and $d_0 = d_0(\eps_1, \eps_2)$ be large.

For $i \in [2]$, choose $v_i' \in B_{G-V(\cP)+v_i}^r(v_i)$ and $F_i' \subseteq (F_i \setminus B_{G-V(\cP)+v_i}^r(v_i)) \cup \{v_i'\}$ such that $F_i'$ is an $(\exp((\log \log n)^{80}), m)$-expansion of $v_i'$ in $F_i$.
This is possible because 
since $F_i$ is connected, $F_i \setminus B_{G-V(\cP)+v_i}^r(v_i)$ has at most $|B_{G-V(\cP)+v_i}^r(v_i)|+1$ components. 
By Pigeonhole Principle, one of them has size at least $|F_i \setminus B_{G-V(\cP)+v_i}^r(v_i)|/(|B_{G-V(\cP)+v_i}^r(v_i)|+1) \ge (|F_i|-|V_i|)/(|V_i|+1) \ge (\exp((\log \log n)^{90})-\exp((\log \log n)^{7}))/(\exp((\log \log n)^{7})+1) > \exp((\log \log n)^{80})$.
For $i \in [2]$, let $L_i$ be a shortest path from $v_i$ to $v_i'$ in $B_{G-V(\cP)+v_i}^r(v_i)$.
Note that in fact $\ell(L_i) = r < m$ by our choice.

Let $U' = U \cup V(F_1) \cup V(F_2) \cup V_1 \cup V_2$.
So $|U'| = |U| + |V(F_1)| + |V(F_2)| + |V_1| + |V_2| \le \exp((\log \log n)^{10}) + \exp((\log \log n)^{90}) + \exp((\log \log n)^{90}) + \exp((\log \log n)^{7}) + \exp((\log \log n)^{7}) \le \exp((\log \log n)^{100})$. 
By Lemma~\ref{lem:large-adjuster} with $(G,U,D,r)_{\ref{lem:large-adjuster}} = (G,U',\exp((\log \log n)^{100}),30m)$, there is a $(\exp((\log \log n)^{100}),m,30m)$-adjuster, and thus a $(\exp((\log \log n)^{80}),m,30m)$-adjuster, say $\cA=(v_3,F_3,v_4,F_4,A)$, in $G-U'$ with length $\ell(\cA)\le |A|+1\leq 600m^2$. Let $\bar{\ell}=\ell-20m-\ell(\cA)-\ell(L_1)-\ell(L_2)$,
so that $0\leq \bar{\ell}\le n/\log^{10}n$. As $|A\cup U \cup B_{G-V(\cP)+v_1}^r(v_1) \cup B_{G-V(\cP)+v_2}^r(v_2)|\leq 600m^2+\exp((\log \log n)^{10})+ \exp((\log \log n)^{7}) + \exp((\log \log n)^{7}) \leq \exp((\log \log n)^{80})/\log^3 n$, by Lemma~\ref{longconnect4} with $(G,D,A)_{\ref{longconnect4}} = (G,\exp((\log \log n)^{80}),A \cup U \cup B_{G-V(\cP)+v_1}^r(v_1) \cup B_{G-V(\cP)+v_2}^r(v_2) \setminus \{v_1',v_2'\})$, there are paths $P$ and $Q$ in $G-(A \cup U \cup B_{G-V(\cP)+v_1}^r(v_1) \cup B_{G-V(\cP)+v_2}^r(v_2) \setminus \{v_1',v_2'\})$ which are vertex disjoint, both connect $\{v_1',v_2'\}$ to $\{v_3,v_4\}$ and so that $\bar{\ell}\leq \ell(P)+\ell(Q)\leq \bar{\ell}+20m$.
Note that we can assume, without loss of generality, that $P$ is a $v_1',v_3$-path and $Q$ is a $v_2',v_4$-path.

Now, $0\leq \ell-\ell(P)-\ell(Q)-\ell(\cA)-\ell(L_1)-\ell(L_2)\leq 20m$. 
As $\cA$ is a $(\exp((\log \log n)^{80}),m,30m)$-adjuster, there is a $v_3,v_4$-path in $G[A\cup \{v_3,v_4\}]$ with length $\ell(\cA)$, and therefore $\ell(\cA)= \pi(v_3,v_4,G)\mod 2$. 
Then, as $\ell(L_1)=\pi(v_1,v_1',G)\mod 2$, $\ell(L_2)=\pi(v_2,v_2',G)\mod 2$, $\ell(P)=\pi(v_1',v_3,G)\mod 2$, $\ell(Q)=\pi(v_2',v_4,G)\mod 2$, $\ell=\pi(v_1,v_2,G)\mod 2$ and $\pi(v_1,v_2,G)=\pi(v_1,v_1',G)+\pi(v_1',v_3,G)+\pi(v_3,v_4,G)+\pi(v_4,v_2',G)+\pi(v_2',v_2,G) \mod 2$, we have $\ell-\ell(P)-\ell(Q)-\ell(\cA)-\ell(L_1)-\ell(L_2)=0\mod 2$. That is, there is some $i\in \N$ with $2i=\ell-\ell(P)-\ell(Q)-\ell(\cA)-\ell(L_1)-\ell(L_2)$, where $i\leq 10m$.

Therefore, by the definition of adjuster, there is a $v_3,v_4$-path $R$  with length $\ell(\cA)+2i=\ell-\ell(P)-\ell(Q)-\ell(L_1)-\ell(L_2)$ in $G[A\cup\{v_3,v_4\}]$. Then, $P' = L_1 \cup P\cup R\cup Q \cup L_2$ is a $v_1,v_2$-path with length $\ell$ in $G-U$.
Moreover, $\cP_i, L_i$ is an ordering of consecutive shortest paths from $v_i$ in $V_i$ for $i \in [2]$.
By construction, $V(P \cup R \cup Q) \setminus \{v_1',v_2'\}$ is disjoint from $B_{G-V(\cP)+v_i}^r(v_i)$ for $i \in [2]$.
This completes the proof.
\end{proof}

Now we are ready to show the following.

\begin{lemma} \label{lem:sparse-no-large-deg}
There exists $\eps_1 > 0$ such that for any $0 < \varepsilon_2 < 1/5$ and $s \ge 20$,
there exist $d_0 = d_0(\varepsilon_1, \varepsilon_2, s)$ and a constant $0 < t_4 < 1/3$ such that the following holds for each $n \ge d \ge d_0$ and $d < 2 \log^{s} n$.
Suppose that $G$ is a $\tk_{d/2}^{(2)}$-free $n$-vertex bipartite $(\varepsilon_1, \varepsilon_2 d)$-expander subgraph $G$ with $\delta(G) \ge d$ and $\Delta(G) \le d^2 \log^{10} n$. 
Then $G$ contains a $\tk^{(\ell)}_{t_4 d}$ for some $\ell \in \mathbb{N}$.
\end{lemma}

\begin{proof}
Let $\eps_1 > 0$ be such that Lemma \ref{lem-finalconnect-sparse2} holds and $d_0 = d_0(\eps_1, \eps_2, s)$ be large. 
Let $0<t_4<1/3$ be such that $t = \lceil t_4 d \rceil$, $l = 2 \lceil \log^7 n \rceil$, $r_1 = \lceil (\log \log n)^5 \rceil$ and $r_2 = \lceil \log n/(300s \log \log n) \rceil$. 

By Lemma \ref{lem-sparse-far-apart} with $s_{\ref{lem-sparse-far-apart}} = s$, let $v_1,\cdots,v_{2t}$ be $2t$ distinct vertices that are at distance at least $\log n/(50s \log \log n) > 3r_2$ apart.
At least half of them are in the same partition of $G$.
Without loss of generality, let $v_1,\cdots,v_{t}$ be $t$ vertices in the same partition of $G$ that will serve as $t$ core vertices in the final balanced clique subdivision construction.
By Lemma \ref{lem-sparse-small-ball} with $q=0$, for each $i \in [t]$, let $V_i = B^{r_1}_G(v_i)$ be the ball of radius $r_1$ around $v_i$.
Note that $|V_i| \le \Delta(G)^{r_1} \le (d^2 \log^{10} n)^{(\log \log n)^5+1} < \exp((\log \log n)^7)$.

Let $K \subseteq {[t] \choose 2}$ be maximal such that there exists a family of pairwise internally disjoint paths $\mathcal{P} = \{P_k: k \in K\}$ such that
\stepcounter{propcounter}
\begin{enumerate}[label = {\bfseries \Alph{propcounter}\arabic{enumi}}]
\item For each $\{i,j\} \in K$, $P_{\{i,j\}}$ is a $v_i,v_j$-path of length $\ell$. $P_{\{i,j\}}$ is disjoint from $V_q$ for $q \in [t] \backslash \{i,j\}$.
\item For each $i \in [t]$, there is some ordering of $\cP(v_i,r_1) := \{P_k(v_i, r_1): k \in K, i \in k\}$, that is, all the subpaths of length $r_1$ with one end vertex $v_i$ of the paths in $\mathcal{P}$ incident to $v_i$, so that they form consecutive shortest paths from $v_i$ in $V_i$.
\item For each $i \in [t]$, $(V(\cP) \setminus V(\cP(v_i,r_1))) \cap B^{r_1}_{G - V(\cP(v_i,r_1)) + v_i}(v_i) = \emptyset$.
\end{enumerate}

One can verify that if $K = {[t] \choose 2}$, then the graph formed by all the paths in $\mathcal{P}$ is a desired $\tk^{(\ell)}_t$. 
Hence, we may assume that there exist distinct $i,j \in [t]$ such that $\mathcal{P}$ contains no such $v_i,v_j$-path of length $\ell$.

Let $W = V(\cP) \backslash \{v_i,v_j\}$.
So $|W| \le |K| \ell < t^2 \ell/2 \le d^2 \log^7 n$.
By Lemma \ref{lem-sparse-small-ball} with $\{P_1,\dots,P_q\}_{\ref{lem-sparse-small-ball}} $ $= \cP(v_i,r_1)$, we have $|B_{G - V(\cP(v_i,r_1)) + v_i}^{r_1}(v_i)| \ge d^2 \log^{10} n$.
Since $(V(\cP) \setminus V(\cP(v_i,r_1))) $ $\cap B^{r_1}_{G - V(\cP(v_i,r_1)) + v_i}(v_i)$ $ = \emptyset$,
we have $B_{G - W}^{r_1}(v_i) = B_{G - V(\cP(v_i,r_1)) + v_i}^{r_1}(v_i)$ and $|B_{G - W}^{r_1}(v_i)| \ge d^2 \log^{10} n$.
By Lemma \ref{lem-sparse-large-ball} with $(Y,W,s)_{\ref{lem-sparse-large-ball}} = (B_{G - W}^{r_1}(v_i), W, 3s)$,
we have $|B^{r_1 + r_2}_{G-W}(v_i)| \ge \exp( (\log n)^{1/4} ) > \exp((\log \log n)^{90})$.
Similarly, $B_{G - W}^{r_1}(v_j) = B_{G - V(\cP(v_j,r_1)) + v_j}^{r_1}(v_j)$ and $|B^{r_1 + r_2}_{G-W}(v_j)| > \exp((\log \log n)^{90})$.
Since for any distinct $p,q \in [t]$, $v_p$ and $v_q$ are at distance at least $3r_2$ by our choice, $B^{r_1 + r_2}_{G-W}(v_i)$ (respectively $B^{r_1 + r_2}_{G-W}(v_j)$) is disjoint from $V_q$ for $q \in [t] \backslash \{i\}$ (respectively $q \in [t] \backslash \{j\}$).  
Let $U = W \cup (\cup_{q \in [t] \backslash \{v_i,v_j\} } V_q)$ and $m = \frac{800}{\eps_1} \log^3 n$.
Hence, there exist vertex disjoint $B^{r_1}_{G-W}(v_i) \subseteq F_i \subseteq B^{r_1 + r_2}_{G-W}(v_i)$ and $B^{r_1}_{G-W}(v_j) \subseteq F_j \subseteq B^{r_1 + r_2}_{G-W}(v_j)$ such that $F_i$ is $(\exp((\log \log n)^{90}), m)$-expansion of $v_i$ in $G - U$ and $F_j$ is $(\exp((\log \log n)^{90}), m)$-expansion of $v_j$ in $G - U$.

Note that $|U| \le |K| \ell + t \exp((\log \log n)^{7}) < t^2 \ell/2 + t \exp((\log \log n)^{7}) <  \exp((\log \log n)^{8})$.
By applying Lemma \ref{lem-finalconnect-sparse2} with $(F_1,F_2,U,\cP)_{\ref{lem-finalconnect-sparse2}} = ( F_i,F_j, U, \cP)$, we have that there is a $v_i,v_j$-path $P_{\{i,j\}}$ of length $\ell$ in $G - U$.
Moreover, previous ordering of $\cP(v_i,r_1)$ and $P_{\{i,j\}}(v_i,r_1)$  (respectively, previous ordering of  $\cP(v_j,r_1)$ and $P_{\{i,j\}}(v_j,r_1)$) are consecutive shortest paths from $v_i$ in $V_i$ (respectively from $v_j$ in $V_j$).
In addition,  $(V(P_{\{i,j\}}) \setminus V(P_{\{i,j\}}(v_q,r_1))  ) \cap B^{r_1}_{G - V(\cP(v_q,r_1)) + v_q}(v_q) = (V(P_{\{i,j\}}) \setminus V(P_{\{i,j\}}(v_q,r_1))  ) \cap B^{r_1}_{G - V(\cP) + v_q}(v_q) = \emptyset$ for $q \in \{i,j\}$.
Therefore, $K \cup \{\{i,j\}\}$ contradicts the maximality of $K$. This completes the proof.
\end{proof}

We can now prove Lemma \ref{lem:sparse}.

\begin{proof}[Proof of Lemma \ref{lem:sparse}]
Let $\eps_1 > 0$ be such that Lemmas \ref{lem:sparse-large-deg} and \ref{lem:sparse-no-large-deg} hold.
Let $d_0 = d_0(\eps_1, \eps_2, s)$ be large and $t_3, t_4$ be constants in Lemmas \ref{lem:sparse-large-deg} and \ref{lem:sparse-no-large-deg} respectively. 

Let $t_2 = \min\{\frac{\eps_1\eps_2}{8 \log^2 (15/2)}, t_3, \frac{1}{3}t_4\}$ and $c=1/4$. 
Let $\Delta = c^2 d^2 \log^{10} n$ and $L = \{v \in V(G): d(v) \ge \Delta\}$.
If $|L| \ge 2t_2 d$, then by Lemma \ref{lem:sparse-large-deg} $G$ contains a $\tk^{(\ell_1)}_{t_2 d}$ for some $\ell_1 \in \mathbb{N}$.
So $|L| < 2t_2 d$ and let $H = G - L$. 

\medskip

\noindent \textbf{Claim.} $H$ is a $\tk_{d/2}^{(2)}$-free bipartite $(\eps_1 /2, \eps_2 d)$-expander satisfying $\delta(H) > d/3$, $|V(H)| > n/3$, $\delta(H) < 2\log^s |V(H)|$ and $\Delta(H) < \delta(H) \log^{10} |V(H)|$.

As $G$ is  $\tk_{d/2}^{(2)}$-free bipartite, $H$ is also $\tk_{d/2}^{(2)}$-free bipartite.
It is easy to see that $\delta(H) \ge \delta(G) - |L| \ge d - 2t_2d > d/3$.
Since $n \ge d$ is large, $|V(H)| \ge |V(G)| - |L| \ge n - 2t_2 d > n/3$.
Note $\delta(H) \le d < \log^s n < 2 \log^s |V(H)|$.
Moreover, we have $\Delta(H) \le \Delta = c^2 d^2 \log^{10} n < c^2 (3\delta(H))^2 \log^{10} (3|V(H)|) < \delta(H)^2 \log^{10} |V(H)|$.

To finish the proof of the claim, we show that $H$ is an $(\eps_1/2, \eps_2 d)$-expander. 
Let $k = \eps_2 d$.
For any set $X$ in $H$ of size $x \ge k/2$ with $x \le |V(H)|/2 \le |V(G)|/2$, we have
$$|N_G(X)|\geq \eps(x,\eps_1, \eps_2 d)\cdot x \ge \eps(\frac{k}{2},\eps_1, \eps_2 d) \cdot \frac{k}{2} \ge \frac{\eps_1}{\log^2 (15/2)} \cdot \frac{\eps_2 d}{2} \ge 4t_2 d \ge 2|L|. $$
So $|N_H(X)| \ge |N_G(X)| - |L|\geq  |N_G(X)|/2 \geq \eps(x,\eps_1, \eps_2 d)/2 \cdot x = \eps(x,\eps_1/2, \eps_2 d) \cdot x$, as required.
\medskip

Finally by Lemma \ref{lem:sparse-no-large-deg}, $H$ contains a $\tk^{(\ell_2)}_{t_2 d}$ for some $\ell_2 \in \mathbb{N}$, and so does $G$.
This completes the proof.
\end{proof}

\section{Proof of Theorem \ref{thm:main} } \label{sec:proof}

Now we show that Theorem \ref{thm:main} follows from Lemmas \ref{lem:dense} and \ref{lem:sparse}.

\begin{proof}[Proof of Theorem \ref{thm:main}]
Let $\varepsilon_1 > 0$ be such that Corollary \ref{cor-expander} and Lemmas \ref{lem:dense} and \ref{lem:sparse} hold and $\varepsilon_2 = 1/10$. 
Let $G$ be a graph with average degree $d(G) \ge d$ and $d_0 = d/8$.
We may assume that $G$ is $\tk^{(2)}_{d/2}$-free.
By Corollary \ref{cor-expander}, $G$ has a bipartite $(\varepsilon_1, \varepsilon_2 d_0)$-expander subgraph $H$ with $\delta(H) \ge d_0$.
So $H$ is also $\tk^{(2)}_{d/2}$-free.

Let $n_0 = |V(H)|$ and $s = \frac{20}{1-2c}$.
If $d_0 \ge \log^{s} n_0$, then by Lemma \ref{lem:dense}, $G$ contains a $\tk^{(\ell)}_{\sqrt{d_0}/(2\log^{10} n_0)}$ for some $\ell \in \mathbb{N}$, 
and thus contains a $\tk^{(\ell)}_{d_0^{c}/2}$.
Otherwise, by Lemma \ref{lem:sparse}, $G$ contains a $\tk^{(\ell)}_{t_2 d_0}$ for constant $t_2 > 0$ and some $\ell \in \mathbb{N}$.
\end{proof}

\section*{Note added after submission}

Recently, the constant $c$ in Theorem \ref{thm:main} is improved to $1/2$ by Gil Fernández, Hyde, Liu, Pikhurko and Wu \cite{FHLPW22}, and Luan, Tang, Wang and Yang \cite{LTWY22} independently.
Note that both proofs use Lemma \ref{lem:sparse}.

\section*{Acknowledgements}

We thank the anonymous referee for their careful reading and suggestions.


\bibliographystyle{abbrv}
\bibliography{balanced_subdivision}

\end{document}